\documentclass[preprint,12pt]{elsarticle}

\usepackage{amssymb}

\usepackage{amssymb}
\usepackage{amsfonts}
\usepackage{amsmath}
\usepackage{amsbsy}
\usepackage{amsxtra}
\usepackage{amscd}
\usepackage{latexsym}
\usepackage{url}
\usepackage{enumerate}
\usepackage[active]{srcltx}
\usepackage{subfigure}

\newtheorem{theorem}{Theorem}
\newtheorem{lemma}{Lemma}

\newtheorem{remark}{Remark}

\newtheorem{problem}{Problem}

\def\argmin{\operatorname{argmin}}
\def\min{\operatorname{Minimize}}
\def\const{\operatorname{subject~to~}}

\DeclareMathOperator{\diag}{diag}
\DeclareMathOperator{\sgn}{sgn}

\newcommand{\lng}{\langle}
\newcommand{\rng}{\rangle}

\newcommand{\R}{\mathbb R}

\newcommand{\f}{\frac}
\newcommand{\ds}{\displaystyle}

\newenvironment{proof}{{\noindent\bf Proof.}}{\hfill$\Box$\\}

\journal{Linear Algebra and its Applications}

\begin{document}

\begin{frontmatter}


\title{Projection onto simplicial cones by Picard's method}


 \author[label1]{Jorge Barrios\fnref{flabel1}}
 \ead{numeroj@gmail.com} 
 \fntext[flabel1]{The author was supported in part by CAPES.}

 \author[label1]{Orizon  P. Ferreira\corref{cor1}\fnref{flabel2}}
 \ead{orizon@ufg.br}
 \cortext[cor1]{Corresponding author.}
 \fntext[flabel2]{The author was supported in part by
FAPEG, CNPq Grants 471815/2012-8,  305158/2014-7 and PRONEX--Optimization(FAPERJ/CNPq).}

 \author[label2]{S\'{a}ndor  Z. N\'{e}meth\fnref{flabel3}}
 \ead{nemeths@for.mat.bham.ac.uk}
 \fntext[flabel3]{The author was supported in part by the Hungarian Research Grant OTKA 60480.}

 \address[label1]{IME/UFG, Campus II- Caixa Postal 131, 
Goi\^ania, GO, 74001-970, Brazil}
 \address[label2]{School of Mathematics, The University of Birmingham, The Watson Building, Edgbaston, Birmingham B15 2TT, United Kingdom}

\begin{abstract}
By using Moreau's decomposition theorem for projecting onto cones, the problem of projecting onto a simplicial cone is reduced to finding the unique solution of a 
nonsmooth system of equations. It is shown that  Picard's method applied to the  system of equations associated  to the  problem of projecting onto a simplicial
cone  generates  a sequence  that converges linearly to the  solution of the  system. Numerical experiments are presented making the comparison between  Picard's and semi-smooth Newton's methods to solve the nonsmooth system associated with the problem of projecting a point onto a simplicial cone.
\end{abstract}


\end{frontmatter}


\section{Introduction} 
The interest in the subject of projection arises in several situations,   having a wide range of applications in pure and applied mathematics such as  Convex Analysis 
(see e.g. \cite{HiriartLemarecal1}),   Optimization (see  e.g. \cite{BuschkeBorwein96,censor07,censor01,Frick1997,scolnik08,ujvari2007projection}), Numerical Linear Algebra 
(see e.g. \cite{Stewart77}), Statistics  (see e.g. \cite{BerkMarcus96,Dykstra83,Xiaomi1998}), Computer Graphics (see e.g. \cite{Fol90} ) and  Ordered 
Vector Spaces (see e.g. \cite{AbbasNemeth2012,IsacNem86,IsacNem92,NemethNemeth2009,Nemeth20091,Nemeth2010-2}).   More specifically, the projection onto a polyhedral 
cone,  which has as a special case the projection onto a simplicial one, is a problem of high impact on scientific community\footnote{see the popularity of the 
Wikimization page Projection on Polyhedral Cone at www.convexoptimization.com/wikimization/index.php/Special:Popularpages}.  The geometric  nature of this problem 
makes it particularly interesting and important  in many areas of science and technology such as   Statistics~(see e.g. \cite{Xiaomi1998}), 
Computation~(see e.g. \cite{Huynh1992}), Optimization (see  e.g.\cite{Morillas2005,ujvari2007projection}) and Ordered Vector Spaces (see e.g. \cite{NemethNemeth2009}). 

The projection onto a  general simplicial cone is difficult and computationally expensive, this problem has been studied e.g. in \cite{AlSultanMurty1992,EkartNemethNemeth2009,Frick1997,MurtyFathi1982,NemethNemeth2009,ujvari2007projection}.  It is a special convex quadratic
program and its  KKT optimality conditions form the linear complementarity problem (LCP) associated with it,  see e.g
\cite{Murty1988,MurtyFathi1982,ujvari2007projection}.  Therefore, the problem of projecting onto  simplicial cones can be solved by  active set
methods \cite{Bazaraa2006,LiuFathi2011,LiuFathi2012,Murty1988} or any algorithms for solving LCPs,  see e.g  \cite{Bazaraa2006,Murty1988} and
special methods based on its geometry,   see e.g \cite{MurtyFathi1982,Murty1988}. Other fashionable ways to solve this problem are based on
the classical von Neumann algorithm (see e.g. the Dykstra algorithm \cite{DeutschHundal1994,Dykstra83,Shusheng2000}). Nevertheless, these 
methods are also quite expensive (see the numerical results in \cite{Morillas2005} and the remark preceding section 6.3 in 
\cite{MingGuo-LiangHong-BinKaiWang2007}).

In this paper we particularize the Moreau's decomposition theorem for simplicial cones. 
This leads to an equivalence between the problem of projecting a point onto a simplicial cone 
and one of finding the unique solution of a nonsmooth system of equations.   We apply  Picard's method to find a unique solution of  the  obtained 
associated system.   Under a mild assumption on the simplicial cone we show that the method  generate a  sequence  that  converges linearly to
the  solution of the  associated  system of equations.   Numerical experiments are presented making the comparison between  Picard's and
semi-smooth Newton's methods for solving the nonsmooth system associated with the problem of projecting a point onto a simplicial cone.

The organization of the paper is as follows. In Section~\ref{sec:int.1}, some notations,  basic results used in the paper and the  statement of
the  problems  that we are interested   are presented, in particular, the problem of projecting onto simplicial cone.  In
Section~\ref{sec:mdtfsc}  we present some results about projection onto simplicial cones. In Section~\ref{sec:ssnm}  we  present two different
Picard's   iterations   for solving  the problem of projecting onto simplicial cone.  In Section~\ref{sec:crresult}  theoretical and numerical
comparisons between  Picard's  methods  and semi-smooth Newton's method for solving  the problem of projecting onto simplicial cone
\cite{FN2014_newton} are provided. Some final remarks are made in Section~\ref{sec:conclusions}. 
\section{Preliminaries} \label{sec:int.1}

Consider $\R^m$ endowed with an orthogonal coordinate system and let 
$\lng\cdot,\cdot\rng$ be the canonical scalar product defined by it. Denote by $\|\cdot\|$ be
the norm generated by $\lng\cdot,\cdot\rng$.  If $a\in\R$ and $x=(x^1,\dots,x^m)\in\R^m$, then denote $a^+:=\max\{a,0\}$, $a^-:=\max\{-a,0\}$ and 
$$ 
x^+:=\left((x^1)^+,\dots,(x^m)^+\right),  x^-:=\left((x^1)^-,\dots,(x^m)^-\right),  |x|:=\left(|x^1|,\dots,|x^m|\right).
$$
For $x\in \R^m$, the vector $\sgn(x)$ will denote a vector with components equal to $1$, $0$ or $-1$ depending on whether the corresponding component of the vector $x$ is positive, zero or negative. We will call a closed set $K\subset\R^m$ a \emph{cone} if the following conditions hold:
\begin{enumerate}
	\item $\lambda x+\mu y\in K$ for any $\lambda,\mu\ge0$ and $x,y\in K$,
	\item $x,-x\in K$ implies $x=0$.
\end{enumerate}
Let \(K \subset \R^m\) be a  closed convex cone. The {\it polar cone}  and the  {\it dual cone} of \( K\) are, respectively,  the sets
\begin{equation} \label{eq:dpd}
 K^\perp\!\!:=\!\{ x\in \R^m\!  \mid\!  \langle x, y \rangle\!\leq\! 0, \forall \, y\!\in\! K\}, \;  K^*\!\!:=\!\{ x\in \R^m\! \mid\!  \langle x, y \rangle\!\geq\! 0, \forall \, y\!\in \!K\}. 
\end{equation}

It is easy to see that $K^\perp=-K^*$.  The set of all $m\times m$ real  matrices is denoted by $\R^{m\times m}$,  $I$ denotes the $m\times m$ identity matrix and $\diag (x)$ will denote a diagonal matrix corresponding to  elements of $x$.

For an   $M \in \R^{m\times m}$  consider the norm defined by 
$$\|M\|:=\max_{x\ne0}\{\|Mx\|~:~ x\in \R^m,  ~ \|x\|=1\},$$  this definition implies
\begin{equation} \label{eq:np}
\|Mx\|\leq \|M\|\|x\|, \qquad \|LM\|\le\|L\|\|M\|, 
\end{equation}
for any $m\times m$ matrices $L$ and $M$. 

Denote $\R^m_+=\{x=(x^1,\dots,x^m)\in\R^m:x_1\ge0,\dots,x^m\ge0\}$ the nonnegative orthant.
Let $A \in \R^{m\times m}$ be a nonsingular matrix. Then, the cone 
\begin{equation} \label{d:sc}
K:=A\R^m_+=\{Ax~:~ x=(x^1,\dots,x^m)\in\R^m, ~x_1\ge0,\dots,x^m\ge0\}, 
\end{equation}
 is called a  \emph{simplicial cone} or \emph{finitely generated cone}.  Let $z\in \R^m$, then  the {\it projection $P_K(z)$ of the point $z$ onto the cone  \( K\)} is defined by 
$$
P_K(z):=\argmin \left\{\|z-y\|~:~ y\in K\right\}.
$$
From the definition of simplicial cone associated with  the  matrix   $A$ this definition is  equivalent to
$$
P_K(z)\!:=\!\argmin\! \left\{\!\frac{1}{2}\|z-Ax\|^2 : x=(x^1,\dots,x^m)\!\in\!\R^m, ~x_1\!\ge0,\!\dots\!,\! \ge x^m\!\ge0\!\right\}\!.
$$
\begin{remark} \label{re:ppo}
It is easy to see that $P_{\R^m_+}(z)=z^{+}$. It is well know that the projection onto a convex set  is continuous and nonexpansive,  in particular, we have  $\|z^+ -w^+\|\leq \|z-x\|$ for all $ ~z, w \in\R^m,$  see \cite{HiriartLemarecal1}. 
\end{remark}
The above remark shows that  projection onto the  nonnegative orthant is an easy problem.  On the other hand, the projection onto a  general simplicial cone is 
difficult and computationally expensive, this problem has been studied e.g. in \cite{AlSultanMurty1992, Frick1997, FN2014_newton, ujvari2007projection,NemethNemeth2009,EkartNemethNemeth2009}.  The statement
of the  problem that we are interested   is:
\begin{problem}[{\bf projection onto a simplicial cone}] \label{prob:pp}
 {\it   Given $A \in \R^{m\times m}$  a nonsingular matrix  and  $z\in \R^m$,  find  the  projection $P_K(z)$ of the point $z$ onto the  simplicial cone \( K=A\R^m_+\).}
\end{problem}
The problem of  projection onto a  simplicial cone has many different formulations which allow us develop  different techniques   for solving them. In the next remark we present some of these formulations. 
\begin{remark}  Let  $A \in \R^{m\times m}$ be   a nonsingular matrix  and  $z\in \R^m$. From the definition of the simplicial cone associated with  the  matrix   $A$ in \eqref{d:sc}, the problem of   projection onto a  simplicial cone   $K=A\R^m_+$ may be  stated equivalently as the following quadratic problem
$$
  \min  ~ \frac{1}{2}\|z-Ax\|^2, \qquad  \const  x \ge 0.
$$
Hence, if $v\in R^m$ is the   unique  solution of this problem then  we have   $P_K(z)=v$.  The above problem is equivalent to the  following nonnegative quadratic problem
\begin{equation}  \label{eq:nap}
 \min  ~ \frac{1}{2}x^\top Qx+x^\top b +c , \qquad  \const  x \ge 0, 
\end{equation}
by taking $Q=A^\top A$,  $b=-A^\top z$ and $c=z^\top z /2$.  
The  optimality condition for the problem \eqref{eq:nap}  implies  that its solution  can be obtained by solving the following linear complementarity problem 
\begin{equation} \label{eq:lcp}
y=Qx+b, \qquad x \ge 0, \qquad    y \ge0, \qquad   x^\top y=0.
\end{equation}
where $y$ is a column vector of variables in $\R^m$.  It is  easy to establish  that corresponding to each  nonnegative quadratic problems  \eqref{eq:nap} and each  linear complementarity problems \eqref{eq:lcp} associated to  symmetric positive definite matrixes, there are equivalent  problems of   projection onto  simplicial cones.   Therefore, the problem of projecting onto  simplicial cones can be solved by  active set
methods \cite{Bazaraa2006,LiuFathi2011,LiuFathi2012,Murty1988} or any algorithms for solving LCPs,  see e.g  \cite{Bazaraa2006,Murty1988} and
special methods based on its geometry,   see e.g \cite{MurtyFathi1982,Murty1988}. Other fashionable ways to solve this problem are based on
the classical von Neumann algorithm (see e.g. the Dykstra algorithm \cite{DeutschHundal1994,Dykstra83,Shusheng2000}). Nevertheless, these 
methods are also quite expensive (see the numerical results in \cite{Morillas2005} and the remark preceding section 6.3 in 
\cite{MingGuo-LiangHong-BinKaiWang2007}).
\end{remark}
As we will see in the next section, by using Moreau's decomposition theorem for projecting onto cones,  solving Problem~\ref{prob:pp} is reduced to solving the following  
problem.
\begin{problem}[{\bf nonsmooth equation}]  \label{prob:ne}
 {\it   Given $A \in \R^{m\times m}$  a nonsingular matrix   and  $z\in \R^m$,  find  the unique  solution $u$ of the nonsmooth equation}
\begin{equation}\label{equation}
		\left(A^\top A-I\right)x^++x=A^\top z.
\end{equation} 
In this case, $P_K(z)=Au^+$ where $K=A\R^m_+$.
\end{problem}
 Since $x^+=(x+|x|)/2$  the Problem~\ref{prob:ne} is equivalent to the following problem:

 \begin{problem}[{\bf absolute value equation}]  \label{prob:nee}
 {\it   Given $A \in \R^{m\times m}$  a  nonsingular matrix   and  $z\in \R^m$,  find  the unique  solution $u$ of the absolute value equation}
\begin{equation}\label{equatione}
		\left(A^\top A+I\right)x+ \left(A^\top A-I\right)|x|=2A^\top z.
\end{equation} 
In this case, $P_K(z)=Au^+$ where $K=A\R^m_+$.
\end{problem}
We will show in Section~\ref{sec:ssnm} that Problem~\ref{prob:ne} and Problem~\ref{prob:nee} can   be solved by using Picard's  method.  We end this section with the Banach's fixed point theorem which will be used for proving our main result, its proof can be found in \cite{Kreyszig1978} (see Theorem~$5.1-2$ pag. 300 and Corollary~$5.1-3$ pag. 302).
\begin{theorem}  [Banach's fixed point theorem] \label{fixedpoint}
Let $({\mathbb X}, d)$ be a non-empty complete metric space, $0\leq \alpha <1$ and $T : {\mathbb X} \to  {\mathbb X}$  a mapping  satisfying $d(T(x),T(y)) \le \alpha d(x,y)$, for all
$ x, y \in {\mathbb X}$. Then there exists an unique $x\in {\mathbb X}$  such that $T(x) = x$. Furthermore, $x$ can be found as follows: start with an arbitrary element $x_0 \in {\mathbb X}$ and define a sequence $\{x_n\}$ by $x_{n+1}= T(x_{n})$, then $\lim_{n  \to +\infty}x_{n}= x$ and the following inequalities hold:
$$
d(x, x_{n+1}) \leq \frac{ \alpha}{1- \alpha} d(x_{n+1},x_n), \quad  d(x, x_{n+1}) \leq \alpha  d(x,x_n),  \quad n=0,1, \ldots.
$$
\end{theorem}
\section{Moreau's decomposition theorem for simplicial cones}  \label{sec:mdtfsc}
In this section we present some results about projection onto simplicial cones.  We recall the following result due to Moreau \cite{Moreau1962}:
\begin{theorem}[Moreau's decomposition theorem]\label{tmor}
Let $K,L\subseteq\R^m$ be two mutually polar cones in $\R^m$. Then, the following 
statements are equivalent:
\begin{enumerate}
	\item $z=x+y,~x\in K,~y\in L$ and $\langle x,y\rangle=0$,
	\item $x=P_K(z)$ and $y=P_L(z)$.
\end{enumerate}
\end{theorem}

\begin{remark} \label{r:pdu}
Let $K$ be  a cone in $\R^m$. Note that from Moreau's  decomposition theorem,  definition of the polar cone and the dual cone in \eqref{eq:dpd}  and the relationship $K^\perp=-K^*$ it follows that 
$$
P_K(z)=z+P_{K^*}(-z),  \qquad \forall ~z\in\R^m.
$$
Hence the  problem of projecting onto $K$ is equivalent to problem of  projecting onto  $K^{*}$.
\end{remark}

\noindent
The following result follows from the definition  of the polar, see  \cite{AbbasNemeth2012}.
\begin{lemma}\label{lad}  
Let $A \in \R^{m\times m}$  be a nonsingular matrix. Then,
\[
(A\mathbb{R}^m_+)^\perp=-(A^\top)^{-1}\mathbb{R}^m_+.
\]
\end{lemma}
The following result has been proved in \cite{AbbasNemeth2012} by using Moreau's decomposition theorem and Lemma~\ref{lad}.

\begin{lemma}\label{pm}
	Let $A \in \R^{m\times m}$  be a nonsingular matrix and $K=A\R^m_+$ the corresponding simplicial cone.
	Then, for any $z\in\R^m$ there exists a unique $x\in\R^m$ such that the following two 
	equivalent statements hold:
	\begin{enumerate}
		\item $z=Ax^+-(A^\top)^{-1}x^-,~x\in\R^m$,
		\item $Ax^+=P_K(z)$ and $-(A^\top)^{-1}x^-=P_{K^\perp}(z)$.
	\end{enumerate}
\end{lemma}
 The  following result is a direct consequence of Lemma \ref{pm}, it shows that solving Problem~\ref{prob:pp} is reduced to solving Problem~\ref{prob:ne}.
\begin{lemma}\label{pm2}
	Let $A \in \R^{m\times m}$  be a nonsingular matrix, $K=A\R^m_+$ the corresponding simplicial cone
	and  $z\in\R^m$ arbitrary. Then, equations \eqref{equation} and \eqref{equatione} have a unique solution $u$ and $P_K(z)=Au^+$, i.e., to
	solve   Problem~\ref{prob:pp} is equivalent to solving either Problem~\ref{prob:ne} or Problem~\ref{prob:nee}.
\end{lemma}
\begin{proof}
 Since $A$ is  an $m\times m$ nonsingular matrix, multiplying by $A^\top$, the equality in item (i) of Lemma \ref{pm} is  equivalently  transformed into 
$$
A^\top Ax^+ - x^-=A^\top z.
$$	
As $-x^-=x-x^+$, the above equality  is equivalent to \eqref{equation}.  Therefore, equation \eqref{equation}  is equivalent to the equation in item (i) of Lemma \ref{pm}. 
Hence, we conclude from   Lemma \ref{pm} that equation \eqref{equation} has a unique solution $u$ and  $P_K(z)=Au^+$.  Since the equations \eqref{equation} and \eqref{equatione} are equivalent the result follows. 
\end{proof}
\section{Picard's Method} \label{sec:ssnm}
In this section we will present two different  Picard's   iterations,   one of them for solving  Problem~\ref{prob:ne} and   the other one for solving  Problem~\ref{prob:nee}.  
\subsection{Picard's Method  for solving Problem~\ref{prob:ne}}
The {\it Picard's method} for solving Problem~\ref{prob:ne}    is formally  defined by 
\begin{equation}\label{eq:pm}
		x_{k+1}=- \left(A^\top A-I\right)x_{k}^+ +A^\top z , \qquad k=0,1,2,\ldots.
\end{equation}
The sequence  $\{x_k\}$ with starting point $x_0 \in \R^m$,  called  the {\it Picard's sequence} for solving Problem~\ref{prob:ne}.  The next theorem provides a sufficient condition for the linear convergence of the Picard's  iteration \eqref{eq:pm}. 	
\begin{theorem} \label{th:cp1}
 Let $A \in \R^{m\times m}$  be a nonsingular matrix, $K=A\R^m_+$ the corresponding simplicial cone
	and  $z\in\R^m$ arbitrary. If 
\begin{equation} \label{eq:asnm}
 \|A^\top A-I\|<1,
\end{equation}
then the Picard's sequence $\{x_k\}$ for solving  Problem~\ref{prob:ne} converges to the unique solution $u$ of equation \eqref{equation} from any starting point $x_0\in \R^m$, $P_K(z)=Au^+$ and the following error bound   holds
\begin{equation} \label{eq:ebp1}
	\|u- x_{k}\| \leq  \frac{\|A^\top A-I\|}{1-\|A^\top A-I\|}  \|  x_{k}- x_{k-1}\|, \qquad \forall ~ k=1,2\ldots .
\end{equation}
Moreover, the sequence $\{x_k\}$ converges linearly to $u$ as follows	
\begin{equation} \label{eq:ebp2}	
	\|u- x_{k+1}\|\leq  \|A^\top A-I\|\|u - x_{k}\|,  \qquad k=0,1,2, \ldots .
\end{equation}	
\end{theorem}
\begin{proof}
Define the function $ F :  \R^m \to \R^m $ as 
\begin{equation}\label{eq:pm1}
		F(x)=- \left(A^\top A-I\right)x^+ +A^\top z.
\end{equation}
Since Remark~\ref{re:ppo} implies  $\| x^+ - y^+\| \leq \|x - y\|$ for all $ x, y \in \R^m$, from  \eqref{eq:pm1}   it easy to conclude that 
$$
		\|F(x)-F(y)\| \leq   \|A^\top A-I\| \|x-y\|, \quad \forall \; x, y \in \R^m.
$$
Therefore, as by  assumption $\|A^\top A-I\|<1 $  we may apply Theorem~\ref{fixedpoint} with ${\mathbb X}=  \R^m$,  $T=F$,  $d(x,y)=\|y-x\|$ for all $ x, y \in \R^m$ and $\alpha=  \|A^\top A-I\| $,   for concluding that  the Picard's Method \eqref{eq:pm}  or equivalently,  the sequence
$$
x_{k+1}=F(x_k),  \qquad k=0,1, \ldots,  
$$
converges to a unique fixed point $u$ of $F$,  which from \eqref{eq:pm1} is the solution of the Problem~\ref{prob:ne}, i.e., 
$$
\left(A^\top A-I\right)u^+ +  u= A^\top z,
$$
and by using Lemma~\ref{pm2} we have $P_K(z)=Au^+$.  Moreover,  Theorem~\ref{fixedpoint} implies that  the inequalities \eqref{eq:ebp1} and \eqref{eq:ebp2} hold.
\end{proof}	
\subsection{Picard's Method  for solving Problem~\ref{prob:nee}}
The {\it Picard's method} for solving Problem~\ref{prob:nee}   is formally  defined by 
\begin{equation}\label{eq:pme}
		\left(A^\top A+I\right)x_{k+1}=-\left(A^\top A-I\right)|x_{k}| +2A^\top z , \qquad k=0,1,2,\ldots.
\end{equation}
The sequence  $\{x_k\}$ with starting point $x_0 \in \R^m$,  called  the {\it Picard's sequence} for solving equation \eqref{equatione} or  for  projecting a point  $z\in \R^m$ onto the simplicial cone $K$.  From now on we will refer this method as {\it Picard 2}. 

Since  $A \in \R^{m\times m}$  is  a  nonsingular matrix we conclude that  $A^\top A$ is symmetric and positive definite. Hence, $A^\top A+I$ is nonsingular. Then for  simplifying  the notations define
\begin{equation} \label{eq:mc}
C:=\left(A^\top A+I\right)^{-1}\left(A^\top A-I\right).
\end{equation}
Let $\lambda_1,   \ldots \lambda_m$ and $\sigma_1,   \ldots \sigma_m$ be the eigenvalues of $A^\top A$ and $C$, respectively.  As $\lambda_i>0$, for $i=1, 2, \ldots m$, it easy to conclude that  
$$
\|C\|=\mbox{max}\left\{ |\sigma_1|,   \ldots |\sigma_m|\right\}<1,  \quad \mbox{where} \qquad \sigma_i=\frac{1-\lambda_i}{\lambda_i+1},  \qquad i=1, 2, \ldots m.
$$
The  next theorem provides the convergence of the Picard's  iteration \eqref{eq:pme}. 	
\begin{theorem} \label{th:cp2}
Let $A \in \R^{m\times m}$  be a nonsingular matrix, $K=A\R^m_+$ the corresponding simplicial cone
	and  $z\in\R^m$ arbitrary. The Picard's sequence $\{x_k\}$ for solving Problem~\ref{prob:nee} is well defined and  converges  to the unique solution $u$ of equation \eqref{equatione} from any starting point $x_0\in \R^m$, $P_K(z)=Au^+$ and the following error bound   holds
\begin{equation} \label{eq:cb3}
	\|u- x_{k}\| \leq  \frac{\|C\|}{1-\|C\|}  \|  x_{k}- x_{k-1}\|, \qquad \forall ~ k=1,2\ldots .
\end{equation}
Moreover, the sequence $\{x_k\}$ converges linearly to $u$ as follows	
\begin{equation} \label{eq:cb4}	
	\|u- x_{k+1}\|\leq  \|C\| \|u - x_{k}\|,  \qquad k=0,1,2, \ldots .
\end{equation}
\end{theorem}	
\begin{proof}
Since the matrix  $A^\top A+I$ is nonsingular, the function $ F :  \R^m \to \R^m $, 
\begin{equation}\label{eq:pm2}
		F(x):=-\left(A^\top A+I\right)^{-1}\left(A^\top A-I\right)|x| +2\left(A^\top A+I\right)^{-1}A^\top z, 
\end{equation}
 is well defined. 
Since    $\||x|- |y|\|\leq  \|x - y\|$ for all $x, y \in \R^m$, from  \eqref{eq:pm2} and  \eqref{eq:mc}  we conclude that 
$$
		\|F(x)-F(y)\| \leq   \|C\| \|x-y\|, \quad \forall \; x, y \in \R^m.
$$
Therefore, as $\|C\|<1 $  we may apply Theorem~\ref{fixedpoint} with ${\mathbb X}=  \R^m$,  $T=F$,  $d(x,y)=\|y-x\|$ for all $ x, y \in \R^m$ and $\alpha=  \|C\| $,   for concluding that  the Picard's Method \eqref{eq:pme}  or equivalently,  the sequence
$$
x_{k+1}=F(x_k),  \qquad k=0,1, \ldots, 
$$
converges to a unique fixed point $u$ of $F$,  which from \eqref{eq:pm2} is the solution  of the Problem~\ref{prob:nee}, i.e., 
$$
\left(A^\top A+I\right)u+ \left(A^\top A-I\right)|u|=2A^\top z,
$$
and by using Lemma~\ref{pm2} we have $P_K(z)=Au^+$. Moreover,   Theorem~\ref{fixedpoint}  implies that the inequalities   \eqref{eq:cb3} and \eqref{eq:cb4}  hold.
\end{proof}

\section{Comparison between   Picard's and Newton's methods} \label{sec:crresult}

In this section theoretical and numerical comparisons of above Picard's  methods and semi-smooth Newton's method studied in ~\cite{FN2014_newton} are provided.  Also  Picard's  method~\eqref{eq:pme} is applied to solve an specific example.

\subsection{Theoretic   comparison}
In this section theoretical  comparisons   between  Picard's  methods  and semi-smooth Newton's method  for solving Problem~\ref{prob:pp}  will be  provided. 

It is shown in \cite{FN2014_newton} that the semi-smooth Newton method  applied to  equation  \eqref{equation},  namely,
\begin{equation}\label{eq:nm}
	\left(\left(A^\top A-I\right)\mbox{diag}(\mbox{sgn}(x_k^+))+I\right)x_{k+1}=A^\top z,   \qquad k=0, 1, 2, \ldots , 
\end{equation}
is always  well defined and   under the assumption 
\begin{equation} \label{eq:asnm1}
 \|A^\top A-I\|<b<\frac{1}{3},
\end{equation}
on  the matrix $A$ defining the simplicial cone  $K=A\R^m_+$, 
the generated  sequence $\{x_k\}$  converges linearly to the unique solution $u$ of Problem~\ref{prob:ne} from any starting point  and, as a consequence of    Lemma~\ref{pm2}  we have $P_K(z)=Au^+$ for any $z\in \R^m$,  which implies that $u$ solves  Problem~\ref{prob:pp}.

Problem~\ref{prob:pp}, i.e.,  the problem of  projecting a point $z\in \R^m$ onto a simplicial cone $K=A\R^m_+$  is equivalent, by
Lemma~\ref{pm2},  to solving either Problem~\ref{prob:ne} or Problem \ref{prob:nee}.   Note that  solving  Problems~\ref{prob:ne}  by  Picard's
method ~\eqref{eq:pm} assumption \eqref{eq:asnm}  on the  matrix $A$ (see Theorem~\ref{th:cp1})  is  less restrictive than assumption
\eqref{eq:asnm1}.   When solving Problem~\ref{prob:nee}  we only need the invertibility  of the matrix $A$ for  Picard's method ~\eqref{eq:pme}
to converge (see Theorem~\ref{th:cp2}). Therefore, Picard's method ~\eqref{eq:pme} is theoretically  more robust than  Picard's method ~\eqref{eq:pm}  and consequently than semi-smooth Newton method \eqref{eq:nm}. In the next section we will present an example, where according to the established theory, only Picard's method ~\eqref{eq:pme} can be applied.

The main  drawbacks of   Picard~\eqref{eq:pme}   and   semi-smooth Newton~\eqref{eq:nm}  is that both  require the solution of a linear system in each iteration which  constitute the largest computational effort of these methods.   Picard's method~\eqref{eq:pm}  do not have to solve a linear system, avoiding  more  complicated calculations,  which is particularly interesting for large scale problems.  We will investigate the efficiency of these methods in  section~\ref{sec:computationalresults1}.

\subsubsection{Example}
\label{sec:example}
Consider the monotone nonnegative cone, which is a simplicial cone $K$ defined by
\begin{equation} \label{d:mncone}
	K:=\left\{ x=(x^1,\dots,x^m)\in\R^m, ~x^1\ge x^2 \ge  \dots \ge x^m\ge0\right\}.
\end{equation} 
The monotone nonnegative cone and the projection onto it occurs in various important practical problems such as the problem of
map-making from relative distance information e.g., stellar cartography (see web page\footnote{ 
{\small \url{www.convexoptimization.com/wikimization/index.php/Projection_on_Polyhedral_Convex_Cone}}} and Section 5.13.2 in \cite{Dattorro2005})
and isotonic regression \cite{GrotzingerWitzgall1984,BurdakovIvanVassilevski2012,NemethNemeth2012b,Guyader-Jegou-Nemeth-2014}. The isotonic regression 
\cite{BarlowBartholomewBremnerBrunk1972,RobertsonWrightDykstra1988,BestChakravarti1990,PuriSingh1990} is a very important topic in statistics with hundreds of papers and
several books dedicated to this topic. This section provides a different view about projecting onto the monotone nonnegative cones via Picard's
method~\eqref{eq:pme} which is related to the iterative theory of bidiagonal and tridiagonal matrices, and the Fibonacci numbers. 
The  dual of the monotone nonnegative cone  is $K^{*}=A\R^m_+$, where  $A\in\R^{m\times m}$ is the  nonsingular matrix   
\[
A=\left(
\begin{array}{rrrrr}
	     1 &    &        & 	      &	\\
	    -1 &  1 &        & 	      &	\\
	       & -1 & 1      & 	      &	\\
	       &    & \ddots & \ddots & \\
	       &    &        &  -1    & 1 
\end{array}
\right), 
\qquad 
A^\top A=\left(
\begin{array}{rrrrr}
	     2 &  -1  &        & 	      &	\\
	    -1 &  2 &     -1   & 	      &	\\
	       & -1 &  \ddots      & 	 \ddots      &	\\
	       &    & \ddots & 2 & -1  \\
	       &    &        &  -1    & 1 
\end{array}
\right).
\] 
From Remark~\ref{r:pdu}, the problem of projecting a point onto $K^{*}$ is equivalent to projecting onto $K$.
Let $\lambda_1,   \ldots \lambda_m$ be the eigenvalues of $A^\top A$.  From   \cite{Allan2008} we have that    the eigenvalues  of  matrix
$A^\top A$ are given by
\begin{equation} \label{eq:eng}
\lambda_i=2+2\cos\left( \frac{2 i \pi}{2m+1} \right), \qquad  i=1, 2, \ldots,  m.
\end{equation}
Hence from  \eqref{eq:eng} we conclude that 
$$
0<\lambda_{i}<4, \qquad       \lim_{m\to \infty}  \lambda_{m}=0, \qquad  \lim_{m\to \infty}  \lambda_{1}=4,  \qquad  \lim_{m\to \infty} \|A^\top A-I\|=3.
$$
Since  $\|A^\top A-I\|>1$ for all $m\geq 2$, for projecting a point onto the cone  $K^{*}$, we can not apply  semi-smooth Newton method studied in
\cite{FN2014_newton} neither  Picard's  iteration~\eqref{eq:pm}. However Picard's method~\eqref{eq:pme} can be used. In order to reduce the computational cost of this method, for solving the linear system involved in each iteration, we suggest the following triangular decomposition
\[
A^\top A+I= \left(
\begin{array}{rrrrr}
	     d_1 &  -1  &        & 	      &	\\
	     &  d_2 &     -1   & 	      &	\\
	       &  &  \ddots      & 	 \ddots      &	\\
	       &    &  & d_{m-1}& -1  \\
	       &    &        &     & d_m 
\end{array}
\right)
\left(
\begin{array}{rrrrr}
	     1 &    &        & 	      &	\\
	    -\frac{1}{d_2} &  1 &        & 	      &	\\
	       & -\frac{1}{d_3}  &  \ddots      & 	      &	\\
	       &    & \ddots & 1 &   \\
	       &    &        &  -\frac{1}{d_m}     & 1 
\end{array}
\right), 
\]
where $d_m=2$, $d_{i}=3-1/d_{i+1}$  for $i=m-1, \cdots, 1$.  Another alternative for solving the linear system would be to compute the matrices 
 $$
 R=\left(A^\top A+I\right)^{-1}\left(A^\top A-I\right), \qquad   S=\left(A^\top A+I\right)^{-1}A^\top.
 $$
By using the recursion formulas for a tridiagonal matrix from   \cite{daFonseca2007}, which are based on the results of 
\cite{FischerUsmani1969,Usmani1994,Usmani1994b}, after some algebraic manipulations and taking into account that $R$ is symmetric we obtain 
\[
R_{ij}=\left\{
\begin{array}{lll}
	\ds-\f{2F_{2i}F_{2m-2j+1}}{F_{2m+1}}                   & \textrm{if} & 1<i<j<m,           \\\\
	\ds\f{F_{2i}F_{2m-2i}-F_{2i-2}F_{2m-2i+1}}{F_{2m+1}} & \textrm{if} & 1<i=j<m,           \\\\
	\ds\f{2F_{2m-2j+1}}{F_{2m+1}}                        & \textrm{if} & 1=i<j<m,           \\\\
	\ds-\f{2F_{2i}}{F_{2m+1}}                            & \textrm{if} & 1<i<j=m,           \\\\
	\ds-\f2{F_{2m+1}}                                    & \textrm{if} & i=1,\textrm{ }j=m, \\\\ 
	\ds\f{F_{2m-2}}{F_{2m+1}}                            & \textrm{if} & i=j=1, m.          
\end{array}
\right.
\]
\[
S_{ij}=\left\{
\begin{array}{lll}
	\ds-\f{F_{2i}F_{2m-2j+2}}{F_{2m+1}}                   & \textrm{if} & i<j,               \\\\
	\ds\f{F_{2j-1}F_{2m-2i+1}}{F_{2m+1}}               & \textrm{if} & 1<j\le i,          \\\\
	\ds\f{F_{2m-2i+1}}{F_{2m+1}}                 & \textrm{if}  & 1=j\le i. \\\\
\end{array}
\right.
\]
where $F_i$ is the Fibonacci sequence defined by $F_0=0$, $F_1=1$ and $F_{i+2}=F_i+F_{i+1}$.

\subsection{Computational results}
\label{sec:computationalresults}

In this section we present two numerical experiments. In the first, numerical comparisons  between  Picard's  methods  \eqref{eq:pm},  \eqref{eq:pme} and semi-smooth Newton's method \eqref{eq:nm} for
solving Problem~\ref{prob:pp} will be provided. In the second one, we study the behavior Picard's  method~\eqref{eq:pme} solving the problem described in Section~\ref{sec:example}. All programs were implemented in MATLAB Version 7.11 64-bit and run on a $3.40 GHz$ Intel Core $i5-4670$ with $8.0GB$ of RAM. All MATLAB codes and generated data of this paper are available in \url{http://orizon.mat.ufg.br/pages/34449-publications}. 

General considerations:
\begin{itemize}
    \item In order to accurately measure the method's runtime for a problem, each of them was solved $10$ times and the runtime data collected. Then, we defined the corresponding  {\it method's  runtime for a problem} as the median of these measurements. 
\item We consider that the method converged to the solution and stopped the execution when, for some $k$, the condition 
$$
\frac{\|u-x_{k}\|}{\|u\|}<RelativeTolerance, 
$$ is satisfied. 
\end{itemize}

\subsubsection{Numerical experiment I}
\label{sec:computationalresults1}
In this experiment, we study the percentage of  problems for which a method was the fastest one (\textbf{efficiency}) to compare them. With the aim that methods \eqref{eq:pm},  \eqref{eq:pme} and \eqref{eq:nm} find solutions on $1000$ generated random test problems of dimension
$m=1000$, we construct the matrix $A$ (defining the simplicial cone  $K=A\R^m_+$) in each problem satisfying the condition \eqref{eq:asnm1}. 

 We assume that  a method is the fastest one for a problem,  if the corresponding runtime is less than or equal to $1.01$ times  {\it the best time} of all methods to find the solution. 

Each test problem was generated as  follows:
\begin{enumerate}
     \item To construct the matrix $A\in \R^{m\times m}$ satisfying the condition (\ref{eq:asnm1}), we first chose a random number $b$ from the standard uniform distribution on the open interval $(0,1/3)$. Then, we chose a random number $\bar{b}$ from the standard uniform distribution on the open interval $(0,b)$. We computed the matrices $S,V$ and $D$, respectively, from the singular value decomposition of a $m\times m$ generated real matrix containing random values drawn
from the uniform distribution on the interval $[-10^6,10^6]$. Finally we computed 
$$
A = S\left(\mbox{sqrt}\left(I+\frac{\bar{b}}{\nu}V\right)\right)D, 
$$
were  $\nu$ is the largest singular value of $V$  and  $\mbox{sqrt}(I+\frac{\bar{b}}{\nu}V)$ is the square root of the  diagonal matrix $I+\frac{\bar{b}}{\nu}V$. 
\item   We chose the solution $u\in \R^{m}$ containing random values drawn from the uniform distribution on the interval $[-10^6,10^6]$ and  computed $z\in \R^{m}$ from equation (\ref{equation}).  Finally we chose a  starting point $x_0\in \R^{m}$ containing random values drawn from the uniform distribution on the interval $[-10^6,10^6]$.
\end{enumerate}

In order to provide information for the analysis of the large test problems set considered, we use the performance profiles (see \cite{DolanMore2002}). The performance profile for a method is the cumulative distribution function for a performance metric. In this case we use the ratio of the method's runtime versus the best runtime of all of the methods as the performance metric. \textbf{Efficiency} can be checked in the value of the profile function at $1$.  

Figure~\ref{fig:1} shows the performance profiles of the three methods for different relative tolerance values. These graphs reveal that Picard's method  \eqref{eq:pm} was the most  efficient  for low and medium accuracy, while semi-smooth Newton's method (\ref{eq:nm}) was the most efficient for high accuracy requirements.  However, since semi-smooth Newton's method (\ref{eq:nm}) requires at each step the solution of a system of linear equations, which may become unreasonably expensive computationally as the problem dimension increases,  these results suggest that for large scale problems Picard's  method (\ref{eq:pm}) is recommended.

\begin{figure}[ht]
     \begin{center}
        \subfigure[RelativeTolerance=$10^{-7}$]{\includegraphics[width=0.3273\textwidth]{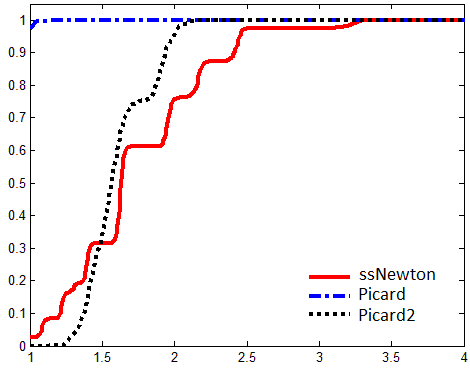}}
        \subfigure[RelativeTolerance=$10^{-10}$]{\includegraphics[width=0.3273\textwidth]{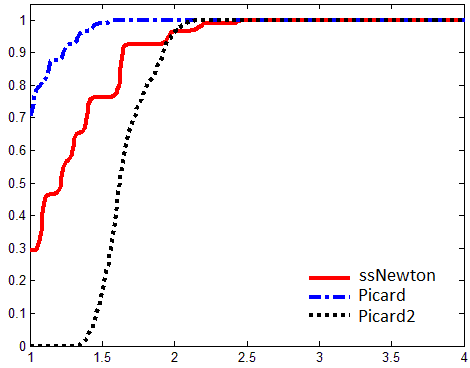}} 
        \subfigure[RelativeTolerance=$10^{-13}$]{\includegraphics[width=0.3273\textwidth]{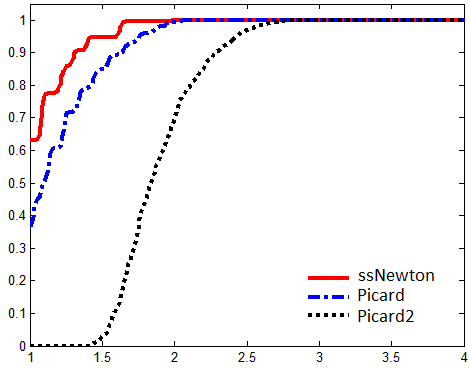}}\\
\end{center}
    \caption{Performance profiles on [1,4] for different accuracies.  Picard,  Picard2 and ssNewton denotes the methods  \eqref{eq:pm}, \eqref{eq:pme} and \eqref{eq:nm}, respectively.  }
   \label{fig:1}
\end{figure}
\noindent
On the other hand,  Picard's method \eqref{eq:pme}  was always the worst, except in the low accuracy case. It  can be inferred from   Figure~\ref{fig:2}, where   convergence mean time for each problem consumed by  Picard's method \eqref{eq:pme}  is  less than   consumed by semi-smooth Newton's method (\ref{eq:nm}).

Figure~\ref{fig:2} shows, as one would expect, the number of iterations on  semi-smooth Newton method is less than Picard's methods~\eqref{eq:pm} and \eqref{eq:pme} for solving the same set of problems and only for certain tolerance  semi-smooth Newton method consumes less time.

\begin{figure}[ht]
     \begin{center}
        \subfigure{\includegraphics[width=0.4927\textwidth]{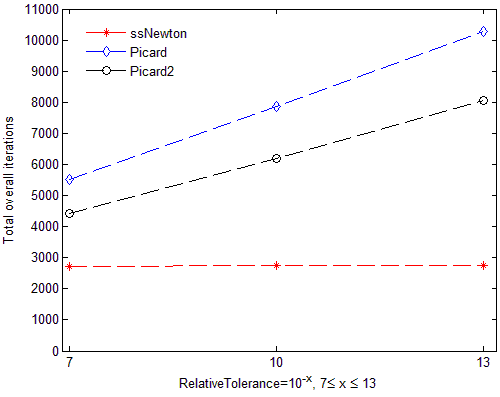}}
         \subfigure{\includegraphics[width=0.4927\textwidth]{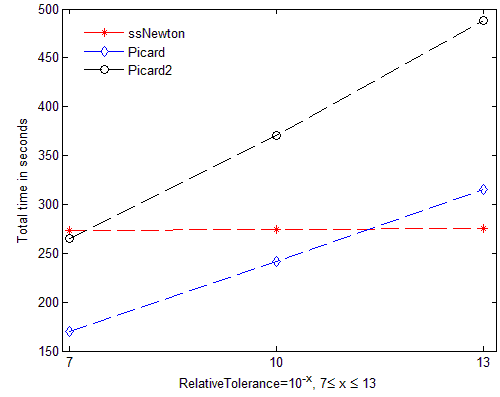}}
     \end{center}
   \caption{Total overall iterations and total time in seconds, performed and consumed, respectively by each method to solve the $1000$ test problems for different accuracies. Picard, Picard2 and ssNewton  denotes the methods  \eqref{eq:pm}, \eqref{eq:pme} and \eqref{eq:nm}, respectively. }
\label{fig:2}
\end{figure}

\subsubsection{Numerical experiment II}
\label{sec:computationalresults2}
In this experiment, we study the behavior of Picard's  method~\eqref{eq:pme} solving the problem described in Section~\ref{sec:example} on sets of $100$ generated random test problems of dimension
$m=100, 500, 1000, 1500, 2000$, respectively. 

Each $m-$dimensional test problem was generated as follows: We constructed the matrix $A$ (defining the simplicial cone  $K^{*}=A\R^m_+$) as is defined in Section~\ref{sec:example}. We chose the solution $u\in \R^{m}$, computed $z\in \R^{m}$ and chose a starting point $x_0\in \R^{m}$ as we described in the previous Section \ref{sec:computationalresults1}.

The computational results obtained are reported in Table~\ref{tab:example1}. From these, it can be noted that for the same dimension, to achieve higher accuracy, the method needs to perform a greater number of iterations and consequently consume more runtime. The same behavior occurs when, for the same accuracy, the dimension of the problem increases.

\begin{table}[htbp]
\centering
\renewcommand{\arraystretch}{1.2}
\begin{tabular}{|c|rrr|rrr|}
\hline
\multicolumn{1}{|c|}{Dimension m }&\multicolumn{3}{c|}{Total Iterations}&\multicolumn{3}{c|}{Total Time}\\\cline{1-7}
 $100$   &4927& 7475& 10036& \multicolumn{1}{|r}{1.096} &  1.624& 2.180\\
 $500$  &6613& 10333& 14055& 66.183& 103.411& 140.812\\
 $1000$  &8120& 12873& 17640& 449.507& 717.310& 984.274\\
 $1500$  &8159& 12924& 17732& 1358.698& 2151.743& 2952.247\\
 $2000$ & 8814& 14054& 19359& 3098.215& 4955.041& 6820.121\\\hline\hline
\multicolumn{1}{|c|}{Relative Tolerance}&\multicolumn{1}{c|}{$10^{-7}$}&\multicolumn{1}{c|}{$10^{-10}$}&\multicolumn{1}{c|}{$10^{-13}$}&\multicolumn{1}{c|}{$10^{-7}$}&\multicolumn{1}{c|}{$10^{-10}$}&\multicolumn{1}{c|}{$10^{-13}$}\\
\hline
\end{tabular}
\caption{Total overall iterations and total time in seconds, performed and consumed, respectively by Picard's method~\eqref{eq:pme} to solve the $100$ test problems  of each dimension for different accuracies.}
\label{tab:example1}
\end{table}

\section{Conclusions} \label{sec:conclusions}
In this paper we studied  the problem of   projection onto a  simplicial cone  which,  via  Moreau's decomposition theorem for projecting onto cones,  is reduced to 
finding the unique solution of a  nonsmooth system of equations.  Our main results show that,   under a mild assumption on the simplicial cone,   we can  apply Picard's method for finding a unique solution of  the  obtained  associated system and  that  the generated sequence converges linearly to the  solution   for  any starting point.  Note that in  Theorem~\ref{th:cp2} we do not make any assumption on the simplicial cone, on the other hand,  we have to solve a linear equation in each iteration.  It would be interesting to see whether the used technique can be applied  for finding  the projection onto  more general cones.   As has been shown 
in \cite{ujvari2007projection}, the problem of projection onto a   simplicial  cone is reduced to a certain type of linear complementarity problem (LCP).   Numerical comparisons   between  Picard's  methods  (\ref{eq:pm},\ref{eq:pme})   and semi-smooth Newton's method \eqref{eq:nm} for solving Problem~\ref{prob:pp} was  provided in Section~\ref{sec:crresult}. It would also be interesting  to compare these   methods with  the  methods proposed in 
\cite{EkartNemethNemeth2009,Morillas2005,ujvari2007projection} and  the Lemke's method for LCPs.

\section*{References}
\bibliographystyle{elsarticle-num}
\bibliography{simcoprojP}

\end{document}